\documentclass{article}
\usepackage{amsthm}
\usepackage{amsmath}
\usepackage{amssymb}
\usepackage{hyperref}
\usepackage{ytableau}
\usepackage{tabularx}
\usepackage{xcolor}
\usepackage{multirow}
\usepackage{comment}

\theoremstyle{theorem}
\newtheorem{theorem}{Theorem}[section]

\theoremstyle{definition}
\newtheorem{definition}[theorem]{Definition}
\newtheorem{example}[theorem]{Example}

\theoremstyle{remark}
\newtheorem{remark}[theorem]{Remark}

\title{
    Algorithms and data structures
    for numerical computations
    with automatic precision estimation
}
\author{
    Igor V. Netay
    \thanks{Joint Stock "Research and production company ``Kryptonite"}
    \thanks{Institute for Information Transmission Problems, Russian Academy of Sciences}
    \href{mailto:i.netay@kryptonite.ru}{i.netay@kryptonite.ru}
}
\date{}

\begin{document}

\maketitle

\begin{abstract}
    We introduce data structures and algorithms to count numerical inaccuracies
    arising from usage of floating numbers described in IEEE~754.
    Here we describe how to estimate precision for some collection of functions
    most commonly used for array manipulations and training of neural networks.
    For highly optimized functions like matrix multiplication,
    we provide a fast estimation of precision and some hint how the estimation
    can be strengthened.
\end{abstract}

\section*{Introduction}

Numerical calculations with floating point numbers almost always are inexact.
The simplest known example is
\begin{verbatim}
    0.1 + 0.2 != 0.3,
\end{verbatim}
which holds for almost all programming languages.
This example illustrates the consequences of non-exact conversion
from decimal number literals to binary float number representation.
At the same time, computations are themselves inexact.

Another well-known example shows that the values of roots
of quadratic equation depend on way the solution was computed.
Consider equation:
\begin{equation}
    \label{eq:2}
    x^2 + 1000 x - 2\cdot 10^{-11} = 0.
\end{equation}
If we find the roots as~$\frac{-a \pm \sqrt{D}}{2a}$ for~$D=b^2-4ac$
for~$ax^2+bx+c=0$, then for~\eqref{eq:2} we get one root~$x= - 1000$
with~$53$ exact bits (maximal number for $64$-bit float) and other
root~$\approx 5.7\cdot 10^{-14}$ without any exact bits.
If we will compute the roots in other way:
\begin{align}
    x_1 &= \frac{-b - \operatorname{sgn}(b)\sqrt{D}}{2a}, \\
    x_2 &= \frac{c}{a\cdot x_1}.
\end{align}
Then we get~$x_1=-1000$ and~$x_2=2\cdot 10^{-14}$ with all $53$~bits
exact for both roots.

The example above shows that accuracy of the result depends on way
of computation frequently.
Actually, usually one cannot easily guess a proper way to get maximally exact results.
Moreover, one cannot get exact results for any arithmetic operation.
Subtraction of numbers that are close to each other is the worst possible case.
During this operation accuracy of the result drops drastically and there is no way to avoid it.

Two mathematically equivalent ways to compute the same expression lead to
different numerical results, it may be unclear which one is closer to correct
mathematical value.
This fact was the motivation for us to implement means to estimate loss of precision
in computations and to track exact mantissa bits in the results.
Also, this helps one to avoid wrong conclusions based upon possible
interpretations of numerical noise as meaningful data.

Calculations with huge arrays of numbers and arithmetic operations
involved like training of neural networks usually lead to high
accumulated digital noise in the result (see~\cite{fluctuations}).
Actually, even simpler numerical computations can lead
to inexact (and unreliable) results frequently.
For instance, among such computations are inversion of ill-conditioned matrices and
solutions of ill-conditioned systems of linear equations.

We do not provide new ways to invert matrix or to solve
linear systems or to perform any other computations.
Instead, we enhanced standard library algorithms with additional
digital error computation and its tracking that will automagically
indicate whether the result is reliable or it should be discarded
and more precise algorithms (or different computational approaches)
are needed.

Usual goal of numerical analysis problem is to design algorithms to construct
maximally or enough precise approximations of mathematical objects.
Now for many applied problems numerical precision is ignored, and computations
are constructed from some standard functions provided by well-known libraries.
Here we propose the answer for the problem of inaccuracy estimation for a given
model and algorithm.
For practical potentially unreliable algorithm calculation of precision
should be performed and may lead to searching for some more reliable algorithms.

\section{Numerical preliminaries}
\label{sec:num-pre}

\subsection{Errors}
\label{subsec:err}

Suppose we calculate an approximation~$\hat{x}$ of real value~$x$.

\begin{definition}
    \textit{Absolute error} is defined as $|\Delta x| = |\hat{x} - x|$.
\end{definition}

\begin{definition}
    \textit{Relative error} is defined as $\delta$
    such that~$\dfrac{\hat{x}}{x}= 1 + \delta$.
\end{definition}

We will denote by \texttt{float} numeric type the standard $64$-bit
floating point numbers (like it is called in Python language and
not like C where it denotes $32$-bit float).
We implement our library in Rust and therefore use its notation
where floating types are called \texttt{f32} and \texttt{f64}.

We extend standard \texttt{float} (equal to Rust~\texttt{f64}) type
and call it~\texttt{xf64} (extended float with $64$ bits).

\begin{example}
    Consider equation~\eqref{eq:2}.
    One can see in~\autoref{tab:2_1},~\autoref{tab:2_2} that computation with~\texttt{xf64}
    maintains the same number of mantissa bits ($53$) and additionally shows 
    precision of computation.
    Method \texttt{\_\_str\_\_} for \texttt{xf64} was modified to substitute 
    unreliable digits of the number with question marks.

    \begin{table}[h]
        \centering
        \begin{tabular}{|l|l|l|}\hline
                type &
                $x_1$, $15$ digits &
                $x_1$, exp.~form
            \\ \hline\hline
                \texttt{float} (64-bit) &
                \texttt{-1000.000000000000} &
                \texttt{-1.000000000000000e+03}
            \\ \hline
                \texttt{xf64} &
                \texttt{-1000.00000000000}\textbf{?} &
                \texttt{-1.00000000000000?e+03}
            \\ \hline\hline
                type &
                $x_2$, $15$ digits &
                $x_2$, exp.~form
            \\ \hline\hline
                \texttt{float} (64-bit) &
                \texttt{0.0000000000000}\textbf{5}\texttt{7} &
                \textbf{5}\texttt{.684341886080801e-14}
            \\ \hline
                \texttt{xf64} &
                \texttt{0.0000000000000}\textbf{?}\texttt{?} &
                \textbf{?}\texttt{.???????????????e-14}
            \\ \hline
        \end{tabular}
        \caption{Solution in the first way}
        \label{tab:2_1}
        \begin{tabular}{|l|l|l|}\hline
                type & 
                $x_2$, $15$ digits &
                $x_2$, exp.~form
            \\ \hline\hline
                \texttt{float} (64-bit) & 
                \texttt{0.000000000000020} &
                \texttt{2.000000000000000e-14}
            \\ \hline
                \texttt{xf64} &
                \texttt{0.000000000000020} &
                \texttt{2.00000000000000}\textbf{?}\texttt{e-14}
            \\ \hline
        \end{tabular}
        \caption{Solution in the second way (same for $x_1$)}
        \label{tab:2_2}
    \end{table}
\end{example}

\subsection{Error propagation in function computation}
\label{subsec:err-fn}

Given an approximation~$\hat{x}$ of real value~$x$,
we approximate~$y=f(x)$ with some~$\hat{y}$.
For estimation of the absolute error in~$\hat{y}$,
linear error estimation is usually applied:
\[
    |\Delta y| =
    |\hat{y} - y| =
    \left|f(x + \Delta x) - f(x)\right| \approx
    \left|\frac{df(x)}{dx}\right|_{x=\hat{x}} \cdot |\Delta x|.
\]
For the relative error we obtain
\[
    \left| \frac{\hat{y} - y}{\hat{y}} \right| \approx
    \left| \frac{\hat{x}}{f(\hat{x})}f'(\hat{x}) \right|
\]
The number $\left| \dfrac{\hat{x}}{f(\hat{x})}f'(\hat{x}) \right|$ is called
the \textit{condition number} $\mathcal{C}_f(\hat{x})$ of function~$f(x)$ at~$\hat{x}$.
It corresponds to the error propagation coefficient of function~$f(x)$ at~$\hat{x}$.
This means that relative error~$\delta$ in~$\hat{x}$ leads to error
of magnitude~$\mathcal{C}_f(\hat{x})\cdot \delta$ in~$f(\hat{x})$.

It is easy to check that if~$\hat{x}$ has~$s$ significant bits,
then~$f(\hat{x})$ has~$s - \log_2\mathcal{C}_f(\hat{x})$ exact bits.
Examples of some condition numbers for some standard floating point
number operations are listed in~\autoref{tab:cn}.

\begin{table}
    \centering
    \begin{tabular}{|c|c|} \hline
        $f(x)$ & $\mathcal{C}_f(x)$ \\ \hline\hline
        $x + a$ & $\dfrac{x}{x + a}$ \\ \hline
        $ax$ & $1$ \\ \hline
        $1/x$ & $1$ \\ \hline
        $x^n$ & $|n|$ \\ \hline
        $\ln(x)$ & $\left|\frac{1}{\ln x}\right|$ \\ \hline
        $\sin x$ & $|x\cot x|$ \\ \hline
        $\cos x$ & $|x\tan x|$ \\ \hline
    \end{tabular}
    \caption[Condition numbers for standard functions]{Condition numbers for some standard operations}
    \label{tab:cn}
\end{table}

\begin{example}
    Function~$\arcsin$ has a poles at~$\pm1$.
    So, near to poles we can see high precision loss:
\begin{verbatim}
    arcsin(0.999999) = 1.56?
    arcsin(0.999???) = 1.???
\end{verbatim}
    Here we calculate the function at~$0.999999$ with~$6$ and~$3$ exact digits.
    High condition number leads to high precision loss.
\end{example}

\subsection{Other reasons of results divergence}

Another important reason of variations in results of numerical computations is
lack of associativity for basic operations like addition and multiplication
for floating point numbers used in programmatic computations
(unlike computations with real numbers~$\mathbb{R}$).

Therefore, any action influencing the order of arithmetic operations
can change the result of the entire computations.
For instance,
\begin{itemize}
    \item vectorization instructions in CPU,
    \item flags used during compilations of imported libraries~%
        (e.\,g.~\texttt{-O2},\texttt{-O3}),
    \item any race condition of any form:
        \begin{itemize}
            \item any use of \texttt{async/await},
            \item multithreading implementation details in your operation system,
            \item networking data transfer implementation details in you cluster.
        \end{itemize}
\end{itemize}

Also, different versions of core systems libraries like \texttt{glibc} can
lead to some small differences and finally lead to some more visible difference.

\section{Related work}

Mainly, other ways to check precision are the following:
\begin{enumerate}
    \item try to redo computation with more bits in floats and find the common
        part of the result;
    \item start with some big enough number of digits/bits and see the loss of
        significance in systems like \texttt{pari/gp} or other libraries
        using \texttt{gmp},
    \item use interval arithmetics (like \texttt{C++ boost/intervals}).
\end{enumerate}

Unfortunately, these ways have the following major issues:
\begin{itemize}
    \item there is no any actual precision guarantee (1),
    \item very low (2) or low (1, 3) performance,
    \item very high (2) or high (1, 3) memory overhead,
    \item the computation should be redone from the beginning and lead to some
        new result of some number data type.
        So, there is not way to check precision of an existing computation.
    \item need to reproduce computation in other programming languages
        (3 for most of cases).
\end{itemize}

We prefer to estimate precision of some existing computation algorithm and track
the estimation of its errors alongside the main dataflow.
Estimated inaccuracies are stored in memory as numbers of exact mantissa bits.

Given initial data with some values of inaccuracies, we can not compute
all the results and their inaccuracies precisely, so we should choose between:
\begin{enumerate}
    \item estimate inaccuracies from below (so parts of results marked as 
        inexact have no sense and should never be interpreted),
    \item estimate inaccuracies from above (so data marked as inexact may have
        some computational errors, while all the other data is exact),
    \item something in the middle without any guaranties.
\end{enumerate}

The most reliable and expensive way is to combine (1) and (2).
In this way mantissa bits are divided into three groups:
\begin{itemize}
    \item ``black'' bits which are noisy and meaningless,
    \item ``white'' bits which are exact and reliable,
    \item ``grey'' bits in the middle part may depend on particular software
        and hardware details, computational process environment and other factors.
\end{itemize}
``Black'' bits arise mainly from data inexactness and mathematical error propagation.
At the same time, ``grey'' bits arise from numerically unstable algorithms.
Actually, some computations evaluate results depending on initial data
discontinuously and cannot be implemented with some numerically stable algorithm.

For instance, computation of eigenvectors for ill conditioned matrix would have
big ``grey'' mantissa parts for any algorithm.

\section{Isolation level for precision estimation}

We can estimate precision of some computation in  different ways.
Given a complex function, we can estimate its common error propagation coefficient
and can also do the same for its parts step by step.

Although bigger fragmentation may provide a better estimation,
some standard functions may have specific software and hardware implementation
in some external libraries.
This imposes a restriction on subdivision level for precision estimation.

For instance, computation of trigonometric functions (usually dependent 
on~\texttt{glibc}) or sum of array (that may be vectorized and dependent on
availability of AVX/SSE instructions in CPU) should be estimated as atomic
numerical operations.
Another difficult atomic operation which is usually dependent on external libraries
is matrix multiplication~(see~\S\ref{sec:mm}).

Deatomization of complex operations like matrix multiplication require
reimplementation of corresponding dependencies like \texttt{blas/cublas/mkl}.
Such very powerful tools can not be easily modified and transferred to other platforms and
hardware settings, so they stay considered as atomic for current work and near future releases
(see~\ref{sec:impl}).

\section{Data structures and basic precision estimations}

Let us introduce numerical data types \texttt{xf64} and others
(like~\texttt{xf32},~\texttt{xf16},~\texttt{xbf16}) consisting of a
\begin{itemize}
    \item floating point number of type \texttt{f64} (or \texttt{f32,f16,bf16})
        representing real value,
    \item number of type~\texttt{u8} (aka~\texttt{byte} or \texttt{unsigned char})
        representing the number of exact bits of the real value.
\end{itemize}
We will call them \textit{extended floating point numbers}.

We explain below how to make arithmetic operations and calculate some
mathematical functions with these numbers.
In all cases we assume that the result of computation with real values
is known and defined in IEEE~754~(see~\cite{IEEE754}), and we provide a way to
estimate precision of this result.

Some special values are zero and \texttt{NaN/Inf}'s.
From the construction, zero should always be maximally exact and \texttt{NaN/Inf}'s
cannot have any exact bits.

\section{Precision estimation for particular functions}

\subsection{Addition, subtraction and sum}

Suppose we add two binary numbers with some given white, grey and black bits.
Then black bits of the result arise from addition of black bits of any of
summands and from alignment to the left of the result.
Estimation for grey bits can be obtained as sum of arguments inaccuracies.

Note that estimation of precision for sum of many numbers may differ from estimation
if they were added one by one due to possible upper rounding for grey bits.

\subsection{Multiplication, division and product}

Estimation of precision for multiplication can be easily deduced to estimation
of precision for addition.
Let us illustrate it by some examples.

For example, take~$6$ with one white, grey and one black bits.
Let us count exact bits of its square~$36$.

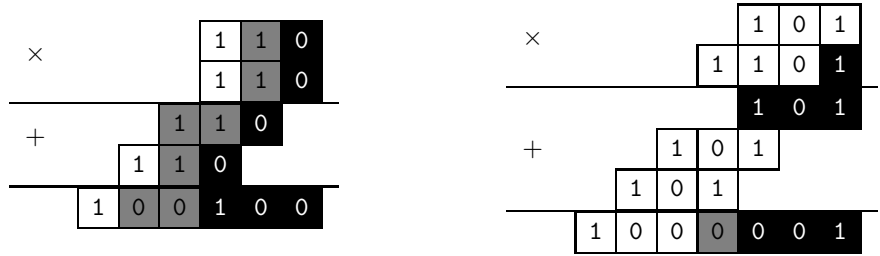
\begin{center}
    \begin{tabular}{cr}
        $\times$
        & $\ytableaushort[\mathtt]{1{*(black!50)1}{*(black){\color{white}0}},1{*(black!50)1}{*(black){\color{white}0}}}$ \\ \hline
        $+$
        & $\ytableaushort[\mathtt]{\none{*(black!50)1}{*(black!50)1}{*(black){\color{white}0}}\none,1{*(black!50)1}{*(black){\color{white}0}}\none\none}$ \\ \hline
        & $\ytableaushort[\mathtt]{1{*(black!50)0}{*(black!50)0}{*(black){\color{white}1}}{*(black){\color{white}0}}{*(black){\color{white}0}}}$ \\
    \end{tabular}
    \hspace*{2cm}
    \begin{tabular}{cr}
        $\times$ &
        $
        \begin{ytableau}[\mathtt]
            \none & 1 & 0 & 1 \\
            1 & 1 & 0 & *(black){\color{white}1} \\
        \end{ytableau}
        $
        \\ \hline
        $+$ &
        $
        \ytableausetup{centertableaux}
        \begin{ytableau}[\mathtt]
            \none & \none & \none & *(black){\color{white}1} & *(black){\color{white}0} & *(black){\color{white}1} \\
            \none & 1 & 0 & 1 \\
            1 & 0 & 1 \\
        \end{ytableau}
        $
        \\ \hline
            &
        $
        \begin{ytableau}[\mathtt]
            1 & 0 & 0 & *(black!50){0} & *(black){\color{white}0} & *(black){\color{white}0} & *(black){\color{white}1}\\
        \end{ytableau}
        $
    \end{tabular}
\end{center}

We see that grey bits can propagate its color to higher bits, but black bits cannot. 
Black bits can induce higher bits to become grey instead.
If we multiply two bits, the color of result is the darkest of arguments colors.

\subsection{Rounding}

Any rounding either keeps the argument the same (in this case precision also
is preserved), or change mantissa bits from some place to the right with some
combination from some discrete predefined set.
If the argument is changed, we should compare the first inexact bit position
and the first changed bit position.
If the first inexact bit position is lower than (to the right from) the first
changed bit position, then the result of rounding have maximal possible number
of exact bits.
Otherwise, the number of exact bits is the same as in the argument.

If the rounding gives zero, it is supposed to be exact.

\subsection{Maxima and minima}

It is well known that floating point numbers do not satisfy axioms of 
linearly ordered set.
Namely, checks of equality fails on the reflexivity rule on special values \texttt{NaN/inf}.

Extended floating point numbers unexpectedly breaks another common expectation
that maximum and minimum of two numbers are always equal to one of the numbers.
Namely, the numeric value is always equal to one of numeric values of arguments,
but the precision may come from the other argument.

For example, for binary numbers~$1.0_2$ and~$1.1_2$ with $2$ and~$1$ exact bits
the minimum is~$1.0$ with~$1$ exact bit.
Indeed, it is not reasonable to set the numeric value of minimum other than
minimal of numerical values of arguments.
At the same time, minimum of~$1.??_2$ and~$1.0?_2$ can not have more than~$1$ exact bit.

\subsection{Differentiable functions}

Estimation of precision for differentiable functions is based on computation of
the corresponding condition number.
The standard floating point functions have explicit derivatives expressed in
elementary functions.
Given the number of exact bits in argument, we can compute
relative error for value returned from the function.
Thus, we can directly obtain estimation of result exact bits.

In this case we consider any such function as an atomic numeric operation which
does not depend on particular implementation.

\subsection{Matrix multiplication and tropical semiring}
\label{sec:mm}

Estimation of matrix multiplication precision is a hard problem, because of
performance issues.
Actually, direct matrix multiplication by the definition is very rarely applied
due to its computation complexity (naive implementation needs~$O(n^3)$ operations for
matrices~$n\times n$).

Let us recall some notation.

\textit{Tropical semi-ring} is a set of numbers and $\infty$ ($-\infty$) 
with two binary operations $\min$ ($\max$) and $+$ with axioms like for ones
for usual numeric ring except group law for $\min$ ($\max$): there is no
inverse operation, so it is called semi-ring instead of ring.

Usually, it is considered with $\infty$ and $\min$, but we need to complement 
this structure with $-\infty$ and $\max$ due to error estimation context.

\textit{Tropical matrix multiplication $\odot$} (also know as
\textit{min-plus matrix multiplication}) is a matrix product over tropical semi-ring.
So, matrix elements can be obtained by usual matrix multiplication with replacement
of $+$ and $\cdot$ with $\max$ and $+$.

Effective tropical matrix multiplication is a difficult problem with many
approaches to particular optimizable cases~(see~\cite{chi2022faster,%
gu2021faster,cslovjecsek2020efficient}).

Unfortunately, our case does not fall into any of these cases.
At the same time, we can estimate tropical matrix product from below
instead of performing its precise numeric computation.

\begin{theorem}
    \label{thm:ex_v1}
    Let $A$,~$B$, and~$C$ are matrices such that~$A \cdot B = C$ of shapes
    ($m \times n$, $n \times k$ and $m \times k$).
    Suppose that~$A$ (resp.~$B$,~$C$) has inaccuracies~$\mathcal{A}$~%
    (resp.~$\mathcal{B}$ and~$\mathcal{C}$).
    Then
    \[
        \mathcal{C} \geqslant \mathcal{A} \odot \mathcal{B}.
    \]
    In particular, the following inequality holds:
    \[
        \mathcal{A} \odot \mathcal{B} \geqslant 
        \frac{1}{n} \cdot \log_2\left(2^\mathcal{A} \cdot 2^\mathcal{B}\right),
    \]
    where power of~$2$ and~$\log_2$ are applied element-wise.
\end{theorem}

\begin{proof}
    We can proof this element-wise over~$C$ and~$\mathcal{C}$.
    So we can put $m=k=1$.

    The needed inequality follows from the inequality between mean and max:
    \begin{equation}
        2^{\mathfrak{c}} \geqslant
        \max(
            2^{\mathfrak{a}_1 + \mathfrak{b}_1},
            \ldots,
            2^{\mathfrak{a}_n + \mathfrak{b}_n}
        )
        \geqslant
        \frac{1}{n}
        \sum_{i=1}^{n}
        2^{
            \mathfrak{a}_1 + \mathfrak{b}_1
        },
        \label{eq:mm_inacc_v1}
    \end{equation}
    where $\mathcal{A}$ is the row~$(\mathfrak{a}_1,\ldots,\mathfrak{a}_n)$,
    and~$B$ is the column~$(\mathfrak{b}_1,\ldots,\mathfrak{b}_n)$,
    and~$\mathcal{C}$ is the $1\times 1$ matrix~$(\mathfrak{c})$.
    This inequality is obvious.
\end{proof}

Actually, this inequality (applied in~\texttt{xnumpy-1.0.0}, see~\S\ref{sec:impl}) 
gives an estimation that can be easily strengthened by the following.

\begin{theorem}
    \label{thm:ex_v2}
    Let $A$,~$B$, and~$C$ are matrices such that~$A \cdot B = C$ of shapes
    ($m \times n$, $n \times k$ and $m \times k$).
    Suppose that~$A$ (resp.~$B$,~$C$) has inaccuracies~$\mathcal{A}$~%
    (resp.~$\mathcal{B}$ and~$\mathcal{C}$).
    Then
    \[
        2^{
            \mathcal{C}
        } \geqslant
        \max\left(
            A \odot 2^\mathcal{B},
            2^\mathcal{A} \odot B
        \right),
    \]
    where power of~$2$ and~$\max$ are applied element-wise.
\end{theorem}

This estimate is stricter~(applied in~\texttt{xnumpy-1.0.1}, see~\S\ref{sec:impl})
and gives more exact estimation of matrix multiplication precision.

\begin{proof}
    The proof can be proceeded in almost the same manner like the proof 
    of~Theorem~\ref{thm:ex_v1} with replacement of the initial precision estimation.

    There we estimate absolute inaccuracy value for product of two numbers from below
    by the product of their absolute inaccuracies.
    We can replace the right-hand side by the maximum of one of multipliers
    multiplied with absolute value of inaccuracy of another one.
    So,~\eqref{eq:mm_inacc_v1} becomes
    \[
        2^{\mathfrak{c}} \geqslant
        \max(
            b_1 \cdot 2^{\mathfrak{a}_1}, a_1 \cdot 2^{\mathfrak{b}_1},
            \ldots,
            b_n \cdot 2^{\mathfrak{a}_n}, a_n \cdot 2^{\mathfrak{b}_n}
        )
    \]
    Remaining part of the proof is similar.
\end{proof}

\begin{remark}
    Estimation for tropical product of matrices is identical for Theorems~%
    \ref{thm:ex_v1} and~\ref{thm:ex_v2}.
    We will propose a stronger estimation in~\S\ref{sec:gmean}.
\end{remark}

This estimation can be performed with three usual matrix multiplications and additional
memory usage for~$\mathcal{A},\mathcal{B}$ and twice for~$\mathcal{C}$
(actually, they can be allocated by a single \texttt{malloc/free}-pair to
avoid excessive heap allocations).

\subsection{Gradients in neural networks}

The most widely used frameworks for neural networks (Tensorflow and Torch) are
based on automatic gradient computation and errors backpropagation.
Gradients are computed as partial derivatives for provided collection of functions
(basic collection can be extended by user-defined functions with gradients).
So, the computation dos not actually involve numeric differentiation.
At the same time, training of neural network until its convergence can lead to 
subtraction of close numbers and very small inexact values of the loss function.
As a result, gradients can become inexact, and the neural network can lose numerical stability.

Calculation of precision for gradients requires calculation of derivatives for 
arithmetic functions and standard collection of trigonometric and other functions.
So, it can be deduced to previously discussed cases.
Gradients of matrix multiplication can be deduced to matrix multiplication
with transposed matrix.
Namely, if we have matrices~$A$ and $B$, and $A\cdot B$ receives gradient~$G$,
then gradients for $A$ and $B$ are equal, respectively,
$
    G \cdot B^{T} \text{ and } A^{T} G.
$
So, if some of matrices $A$ and $B$ is inexact, then its inexactness 
contribute to inexactness in both forward and backward propagation steps.

Precision of convolution layers and their gradients are calculated analogously 
to the same for matrix multiplications.

\section{Hölder's inequality and matrix multiplication precision}
\label{sec:gmean}

\begin{theorem}
    \label{thm:ex_h}
    Let $A$,~$B$, and~$C$ are matrices such that~$A \cdot B = C$.
    Suppose that~$A$ (resp.~$B$,~$C$) has inaccuracies~$\mathcal{A}$~%
    (resp.~$\mathcal{B}$ and~$\mathcal{C}$).
    Then
    \[
        \mathcal{C} \geqslant \mathcal{A} \odot \mathcal{B}.
    \]
    In particular, the following inequality holds for any $p>1$:
    \[
        \mathcal{A} \odot \mathcal{B} \geqslant 
        \frac{1}{np} \cdot \log_2\left(2^{p\mathcal{A}} \cdot 2^{p\mathcal{B}}\right),
    \]
    where power of~$2$ and~$\log_2$ are applied element-wise.
\end{theorem}

The proof is actually the same with replacement of inequality between the mean value
and maximum with inequality between of generalized mean (also known as ``power mean'')
of degree~$p$.

This seems to be trivial due to addition of $p$ into power of $2$ and into 
the division outside of logarithm.
Actually, it is not trivial, because the dot inside the brackets means 
matrix multiplication, so we can not reduce~$p$ in the formula above.

Moreover, there is generalized mean inequality (also known as ``Hölder's 
inequality'')
\[
    M_p(x_1, \ldots, x_n) \leqslant M_q(x_1, \ldots, m_n),
\]
for $p < q$, where
\[
    M_p(x_1, \ldots, x_n) := \sqrt[p]{
        \frac{1}{n}
        \sum_{i=1}^{n} x_i^p
    }.
\]
and the limit property
\[
    \lim_{p\to\infty}M_p(x_1,\ldots,x_n) = \max(x_1,\ldots,x_n).
\]

It seems that taking this inequality for large~$p$ can give a better estimation
of tropical product.
But its application is limited by floating point overflows.

\begin{theorem}
    Let $A$,~$B$, and~$C$ are matrices such that~$A \cdot B = C$ of shapes
    ($m \times n$, $n \times k$ and $m \times k$).
    Suppose that~$A$ (resp.~$B$,~$C$) has inaccuracies~$\mathcal{A}$~%
    (resp.~$\mathcal{B}$ and~$\mathcal{C}$).
    Then
    \[
        2^{
            p\cdot \mathcal{C}
        } \geqslant
        \max\left(
            A \odot 2^{p \cdot \mathcal{B}},
            2^{p \cdot \mathcal{A}} \odot B
        \right),
    \]
    for any $p>1$,
    where power of~$2$ and~$\max$ are applied element-wise.
\end{theorem}

\begin{proof}
    This Theorem is simply a combination of Theorems~\ref{thm:ex_v2}
    and~\ref{thm:ex_h}.
\end{proof}

\section{Implementation}
\label{sec:impl}

We provide an implementation of our approach~(see~\cite{XNumPyGitHub,XNumPyGitFlic}) by extending
widely applied \texttt{NumPy} library (see~\cite{harris2020array}).
There ``black bits'' are estimated, i.\,e.~if bits or digits are shown as inexact
by this library, then they are meaningless.

One exception is precision computation for matrix product.
If we compute matrix product, we rely on its numerical values
provided by libraries like \texttt{blas/cublas/mkl}.
At the same time, precision computation itself requires matrix multiplication.
Due to performance reasons, these matrix multiplications are also computed
by the same libraries, although in some cases it can give inexact matrix elements,
and this can imply sometimes inexact precision estimations.
But if we rely on some bit inexact numerical values, then we can rely on
some bit inexact precision estimation in the same numerical computation.

Actually, some minor inaccuracies for the condition number computation can
also happen, but they are usually can be omitted.

\section{Future work}
\label{sec:future}

All these estimations hold independent upon platform, software and hardware details.
The implementation is already available for CPU and will be extended for GPU.
Also, more advanced algorithms such as linear algebra and statistic methods
will be overloaded to support precision loss tracking automagically.

\section{Acknowledgments}
\label{sec:ackn}

The author is grateful to his Kryptonite colleagues Vasily Dolmatov, 
Dr. Nikita Gabdullin and Dr. Anton Raskovalov for fruitful discussions of topic and results
and for assistance in testing the \texttt{xnumpy} library.

\bibliographystyle{unsrt}
\bibliography{refs}

\begin{thebibliography}{1}

\bibitem{fluctuations}
Igor~V. Netay.
\newblock Influence of digital fluctuations on behavior of neural networks.
\newblock {\em Indian Journal of Artificial Intelligence and Neural Networking
  (IJAINN)}, 3:1--7, December 2022.

\bibitem{IEEE754}
Ieee standard for floating-point arithmetic.
\newblock {\em IEEE Std 754-2019 (Revision of IEEE 754-2008)}, pages 1--84,
  2019.

\bibitem{chi2022faster}
Shucheng Chi, Ran Duan, Tianle Xie, and Tianyi Zhang.
\newblock Faster min-plus product for monotone instances.
\newblock In {\em Proceedings of the 54th Annual ACM SIGACT Symposium on Theory
  of Computing}, pages 1529--1542, 2022.

\bibitem{gu2021faster}
Yuzhou Gu, Adam Polak, Virginia~Vassilevska Williams, and Yinzhan Xu.
\newblock Faster monotone min-plus product, range mode, and single source
  replacement paths.
\newblock {\em arXiv preprint arXiv:2105.02806}, 2021.

\bibitem{cslovjecsek2020efficient}
Jana Cslovjecsek, Friedrich Eisenbrand, Micha{\l} Pilipczuk, Moritz Venzin, and
  Robert Weismantel.
\newblock Efficient sequential and parallel algorithms for multistage
  stochastic integer programming using proximity.
\newblock {\em arXiv preprint arXiv:2012.11742}, 2020.

\bibitem{XNumPyGitHub}
Igor~V. Netay.
\newblock Xnumpy.
\newblock \url{https://github.com/netay/xnumpy}, 2023.

\bibitem{XNumPyGitFlic}
Igor~V. Netay.
\newblock Xnumpy.
\newblock \url{https://gitflic.ru/project/kryptodpi/xnumpy}, 2023.

\bibitem{harris2020array}
Charles~R. Harris, K.~Jarrod Millman, St{\'{e}}fan~J. van~der Walt, Ralf
  Gommers, Pauli Virtanen, David Cournapeau, Eric Wieser, Julian Taylor,
  Sebastian Berg, Nathaniel~J. Smith, Robert Kern, Matti Picus, Stephan Hoyer,
  Marten~H. van Kerkwijk, Matthew Brett, Allan Haldane, Jaime~Fern{\'{a}}ndez
  del R{\'{i}}o, Mark Wiebe, Pearu Peterson, Pierre G{\'{e}}rard-Marchant,
  Kevin Sheppard, Tyler Reddy, Warren Weckesser, Hameer Abbasi, Christoph
  Gohlke, and Travis~E. Oliphant.
\newblock Array programming with {NumPy}.
\newblock {\em Nature}, 585(7825):357--362, September 2020.

\end{thebibliography}

\end{document}